\documentclass[12pt]{article}
\usepackage{amsmath,amsthm,amssymb}
\usepackage{amssymb,latexsym}

\newtheorem{theorem}{Theorem}

\newtheorem{lemma}{Lemma}

\newtheorem{remark}{Remark}
\begin{document}
\baselineskip=17pt

\title{\bf On the number of pairs of positive integers
$\mathbf{x, y \leq H}$ such that $\mathbf{x^2+y^2+1}$, $\mathbf{x^2+y^2+2}$
are square-free}
\author{\bf S. I. Dimitrov}

\date{}
\maketitle

\begin{abstract}

In the present paper we show that there exist infinitely many
consecutive square-free numbers of the form
$x^2+y^2+1$, $x^2+y^2+2$.
We also establish an asymptotic formula for the number of pairs of
positive integers $x, y \leq H$ such that $x^2+y^2+1$, $x^2+y^2+2$
are square-free.\\
\textbf{Keywords}: Square-free numbers, Asymptotic formula, Gauss sums.\\
{\bf  2010 Math.\ Subject Classification}:  11L05 $\cdot$ 11N25 $\cdot$  11N37
\end{abstract}

\section{Notations}
\indent

Let $H$ be a sufficiently large positive number.
By $\varepsilon$ we denote an arbitrary small positive number, not the same in all appearances.
The letters $d, h, k, l, q, r, x, y$ with or without subscript will denote positive integers.
By the letters $D_1, D_2$, $H_0$ and $t$ we denote real numbers and by $m, n$ -- integers.
As usual $\mu(n)$ is M\"{o}bius' function and $\tau(n)$ denotes the number of positive divisors of $n$.
Further $[t]$ and $\{t\}$ denote the integer part, respectively, the fractional part of $t$.
Instead of $m\equiv n\,\pmod {d}$ we write for simplicity $m\equiv n\,(d)$.
Moreover $(l,m)$ is the greatest common divisor of $l$ and $m$,
and $(l,m,n)$ is the greatest common divisor of $l$, $m$ and $n$.
The letter $p$  will always denote prime number.
We put $\psi(t)=\{t\}-1/2$. As usual $e(t)$=exp($2\pi it$).
For any odd $q$ we denote by $\left(\frac{\cdot}{q}\right)$  the Jacobi symbol.
For any $n$ and $q$ such that $(n, q)=1$ we denote by $\overline{(n)}_q$
the inverse of $n$ modulo $q$. If the value of the modulus is understood from the
context then we write for simplicity $\bar{n}$.

By $G(q,m,n)$ and $G(q,m)$ we shall denote the Gauss sums
\begin{equation}\label{Gausssums}
G(q,m,n)=\sum\limits_{x=1}^{q}e\left(\frac{mx^2+nx}{q}\right)\,,\quad
G(q,m)=G(q,m,0)\,.
\end{equation}
By $K(q,m,n)$ we shall denote the Kloosterman sum
\begin{equation}\label{Kloosterman}
K(q,m,n)=\sum\limits_{x=1\atop{(x, q)=1}}^{q}e\left(\frac{mx+n\bar{x}}{q}\right)\,.
\end{equation}

\section{Introduction and statement of the result}
\indent

The problem for the consecutive square-free numbers arises in 1932 when Carlitz \cite{Carlitz} proved that
\begin{equation}\label{Carlitz}
\sum\limits_{n\leq H}\mu^2(n)\mu^2(n+1)=\prod\limits_{p}\left(1-\frac{2}{p^2}\right)H
+\mathcal{O}\big(H^{\theta+\varepsilon}\big)\,,
\end{equation}
where $\theta=2/3$.
Formula \eqref{Carlitz} was subsequently improved by Heath-Brown
\cite{Heath-Brown} to $\theta=7/11$ and by Reuss \cite{Reuss} to $\theta=(26+\sqrt{433})/81$.

In 2018 the author \cite{Dimitrov1} showed that for any fixed $1<c<22/13$ there exist infinitely many
consecutive square-free numbers of the form $[n^c], [n^c]+1$.

Recently the author \cite{Dimitrov2} proved that there exist infinitely many
consecutive square-free numbers of the form $[\alpha n]$, $[\alpha n]+1$,
where $n$ is natural and $\alpha>1$ is irrational number with bounded partial quotient
or irrational algebraic number.

Also recently the author \cite{Dimitrov3} showed that there exist infinitely many
consecutive square-free numbers of the form $[\alpha p]$, $[\alpha p]+1$,
where $p$ is prime and $\alpha>0$ is irrational algebraic number.

On the other hand in 2012 Tolev \cite{Tolev2} proved ingeniously that
there exist infinitely many square-free numbers of the form $x^2+y^2+1$.
More precisely he established the asymptotic formula
\begin{equation*}
\sum\limits_{1\leq x, y\leq H}\mu^2(x^2+y^2+1)
=cH^2+\mathcal{O}\left(H^{\frac{4}{3}+\varepsilon}\right)\,,
\end{equation*}
where
\begin{equation*}
c=\prod\limits_{p}\left(1-\frac{\lambda(p^2)}{p^4}\right)
\end{equation*}
and
\begin{equation*}
\lambda(q)=\sum\limits_{1\leq x, y\leq q\atop{x^2+y^2+1\equiv 0\,(q)}}1\,.
\end{equation*}
Define
\begin{equation}\label{GammaH}
\Gamma(H)=\sum\limits_{1\leq x, y\leq H}\mu^2(x^2+y^2+1)\,\mu^2(x^2+y^2+2)
\end{equation}
and
\begin{equation}\label{lambdaq1q2mn}
\lambda(q_1, q_2, m, n)=\sum\limits_{x, y\,:\,\eqref{conditions}}e\left(\frac{m x+n y}{q_1 q_2}\right)\,,
\end{equation}
where the summation is taken over the integers $x, y$ satisfying the conditions
\begin{equation}\label{conditions}
\left|\begin{array}{ccc}
1\leq x, y\leq q_1q_2\\
x^2+y^2+1\equiv 0\,(q_1)\\
x^2+y^2+2\equiv 0\,(q_2)
\end{array}\right..
\end{equation}
We define also
\begin{equation}\label{lambdaq1q2}
\lambda(q_1, q_2)=\lambda(q_1, q_2, 0, 0)\,.
\end{equation}

Motivated by these results and following the method of Tolev \cite{Tolev2}
we shall prove the following theorem
\begin{theorem}\label{Mytheorem}  For the sum $\Gamma(H)$  defined by \eqref{GammaH} the asymptotic formula
\begin{equation}\label{asymptoticformula}
\Gamma(H)=\sigma H^2+\mathcal{O}\left(H^{\frac{8}{5}+\varepsilon}\right)
\end{equation}
holds. Here
\begin{equation}\label{sigmaproduct}
\sigma=\prod\limits_{p}\left(1-\frac{\lambda(p^2, 1)+\lambda(1, p^2)}{p^4}\right)\,.
\end{equation}
\end{theorem}
From Theorem \ref{Mytheorem} it follows that there exist infinitely many
consecutive square-free numbers of the form $x^2+y^2+1$, $x^2+y^2+2$.

\section{Lemmas}
\indent

The first lemma we need gives us the basic properties of the Gauss sum.

\begin{lemma}\label{Gausslemma}
For the Gauss sum we have

\emph{(i)} \quad If $(q_1,q_2)=1$ then
\begin{equation*}
G(q_1 q_2, m_1 q_2+m_2 q_1, n)=G(q_1, m_1 q_2^2, n)\,G(q_2, m_2 q_1^2, n)\,.
\end{equation*}
\quad \emph{(ii)} \quad If $(q,m)=d$ then
\begin{equation*}
G(q, m ,n)=\begin{cases}d\,G\left(q/d, m/d, n/d\right) \;\text{ if }\; d\mid{n}\,,\\
0 \quad\quad\quad\quad\quad\quad\quad\quad \mbox{ if } \; d\nmid{n}\,.
\end{cases}
\end{equation*}
\;\;\;\emph{(iii)} \quad If $(q ,2m)=1$ then
\begin{equation*}
G(q, m, n)=e\left(\frac{-\overline{(4m)}\,n^2}{q}\right)\left(\frac{m}{q}\right)G(q,1)\,.
\end{equation*}
\;\, \emph{(iv)} \quad If $(q, 2)=1$ then
\begin{equation*}
G^2(q, 1)=(-1)^{\frac{q-1}{2}} q\,.
\end{equation*}
\end{lemma}
\begin{proof}
See \cite{Estermann} and \cite{Hua}.
\end{proof}

The next lemma gives us A. Weil's estimate for the Kloosterman sum.

\begin{lemma}\label{Weilsestimate}
\begin{equation*}
|K(q,m,n)|\leq \tau(q)\,q^{\frac{1}{2}}\,(q,m,n)^{\frac{1}{2}}\,.
\end{equation*}
\end{lemma}
\begin{proof}
See \cite{Iwaniec}.
\end{proof}

The next lemma is the central moment in the proof of Theorem \ref{Mytheorem}.
Right here we apply the  properties of the Gauss sum and A.Weil's estimate for the Kloosterman sum.
\begin{lemma}\label{Mylemma1}
Let $8\nmid q_1$, $8\nmid q_2$ and $(q_1, q_2)=1$.
Then for function defined by \eqref{lambdaq1q2mn} the estimation
\begin{equation}\label{Myestimate1}
|\lambda(q_1, q_2, m, n)|\leq 16\,\tau^2(q_1 q_2)\,(q_1 q_2)^{\frac{1}{2}}\,
(q_1 q_2,m,n)^{\frac{1}{2}}
\end{equation}
holds.
In particular we have
\begin{equation}\label{Myestimate2}
\lambda(q_1, q_2)\ll(q_1 q_2)^{1+\varepsilon}\,.
\end{equation}
\end{lemma}

\begin{remark}
An estimate of type \eqref{Myestimate1} is valid for any positive integers $q_1, q_2$.
We introduce the restrictions $8\nmid q_1$, $8\nmid q_2$ since in this case
the proof is slightly simpler and since in our work only such $q_1, q_2$ appear.
\end{remark}

\begin{proof}

\textbf{Case 1.} $2\nmid q_1 q_2$.

Using \eqref{Gausssums}, \eqref{lambdaq1q2mn}, \eqref{conditions} and Lemma \ref{Gausslemma}
we get
\begin{align}\label{lambdaq1q2mnest1}
&\lambda(q_1, q_2, m, n)=\nonumber\\
&=\frac{1}{q_1 q_2}\sum\limits_{1\leq x, y\leq q_1 q_2}e\left(\frac{mx+ny}{q_1q_2}\right)
\sum\limits_{1\leq h_1\leq q_1}e\left(\frac{h_1(x^2+y^2+1)}{q_1}\right)
\sum\limits_{1\leq h_2\leq q_2}e\left(\frac{h_2(x^2+y^2+2)}{q_2}\right)\nonumber\\
&=\frac{1}{q_1 q_2}\sum\limits_{1\leq h_1\leq q_1}e\left(\frac{h_1}{q_1}\right)
\sum\limits_{1\leq h_2\leq q_2}e\left(\frac{2h_2}{q_2}\right)
G(q_1 q_2, h_1 q_2+h_2 q_1, m)\, G(q_1 q_2, h_1 q_2+h_2 q_1, n)\nonumber\\
\end{align}

\begin{align*}
&=\frac{1}{q_1 q_2}\sum\limits_{1\leq h_1\leq q_1}e\left(\frac{h_1}{q_1}\right)\,
G(q_1, h_1 q^2_2, m)\, G(q_1, h_1 q^2_2, n)\nonumber\\
&\times\sum\limits_{1\leq h_2\leq q_2}e\left(\frac{2h_2}{q_2}\right)
G(q_2, h_2 q^2_1, m)\, G(q_2, h_2 q^2_1, n)\nonumber\\
&=\frac{1}{q_1 q_2}\sum\limits_{l_1 | q_1}\sum\limits_{1\leq h_1\leq q_1\atop{(h_1, q_1)=\frac{q_1}{l_1}}}
e\left(\frac{h_1}{q_1}\right)\,G(q_1, h_1 q^2_2, m)\, G(q_1, h_1 q^2_2, n)\nonumber\\
&\times\sum\limits_{l_2 | q_2}\sum\limits_{1\leq h_2\leq q_2\atop{(h_2, q_2)=\frac{q_2}{l_2}}}
e\left(\frac{2h_2}{q_2}\right)G(q_2, h_2 q^2_1, m)\, G(q_2, h_2 q^2_1, n)\,.
\end{align*}
Bearing in mind \eqref{Kloosterman}, \eqref{lambdaq1q2mnest1}, $2\nmid q_1 q_2$ and Lemma \ref{Gausslemma} we obtain
\begin{align}\label{lambdaq1q2mnest2}
\lambda(q_1, q_2, m, n)
&=q_1 q_2\sum\limits_{l_1 | q_1\atop{\frac{q_1}{l_1} | (m, n)}}\frac{1}{l^2_1}
\sum\limits_{1\leq r_1\leq l_1\atop{(r_1, l_1)=1}}
e\left(\frac{r_1}{l_1}\right)\,G(l_1, r_1 q^2_2, m l_1 q^{-1}_1)\, G(l_1, r_1 q^2_2, n l_1 q^{-1}_1)\nonumber\\
&\times\sum\limits_{l_2 | q_2\atop{\frac{q_2}{l_2} | (m, n)}}\frac{1}{l^2_2}
\sum\limits_{1\leq r_2\leq l_2\atop{(r_2, l_2)=1}}
e\left(\frac{2r_2}{l_2}\right)\,G(l_2, r_2 q^2_1, m l_2 q^{-1}_2)\, G(l_2, r_2 q^2_1, n l_2 q^{-1}_2)\nonumber\\
&=q_1 q_2\sum\limits_{l_1 | q_1\atop{\frac{q_1}{l_1} | (m, n)}}\frac{G^2(l_1, 1)}{l^2_1}
\sum\limits_{1\leq r_1\leq l_1\atop{(r_1, l_1)=1}}
e\left(\frac{r_1-\overline{(4 r_1 q^2_2)}(m^2+n^2)l^2_1 q^{-2}_1}{l_1}\right)\nonumber\\
&\times\sum\limits_{l_2 | q_2\atop{\frac{q_2}{l_2} | (m, n)}}\frac{G^2(l_2, 1)}{l^2_2}
\sum\limits_{1\leq r_2\leq l_2\atop{(r_2, l_2)=1}}
e\left(\frac{2r_2-\overline{(4 r_2 q^2_1)}(m^2+n^2)l^2_2 q^{-2}_2}{l_2}\right)\nonumber\\
&=q_1 q_2\sum\limits_{l_1 | q_1\atop{\frac{q_1}{l_1} | (m, n)}}\frac{(-1)^{\frac{l_1-1}{2}}}{l_1}
K(l_1, 1, \overline{4 q^2_2}(m^2+n^2)l^2_1 q^{-2}_1)\nonumber\\
&\times\sum\limits_{l_2 | q_2\atop{\frac{q_2}{l_2} | (m, n)}}\frac{(-1)^{\frac{l_2-1}{2}}}{l_2}
K(l_2, 2, \overline{4 q^2_1}(m^2+n^2)l^2_2 q^{-2}_2)\,.
\end{align}
From \eqref{lambdaq1q2mnest2} and Lemma \ref{Weilsestimate} it follows that
\begin{align}\label{lambdaq1q2mnest3}
|\lambda(q_1, q_2, m, n)|&\leq q_1 q_2\sum\limits_{l_1 | q_1\atop{\frac{q_1}{l_1} | (m, n)}}
\frac{\tau(l_1)}{l^{\frac{1}{2}}_1}\sum\limits_{l_2 | q_2\atop{\frac{q_2}{l_2} | (m, n)}}
\frac{\tau(l_2)}{l^{\frac{1}{2}}_2}\nonumber\\
&\leq q_1 q_2\, \tau(q_1 q_2)\sum\limits_{r_1 | (q_1, m, n)}q^{-\frac{1}{2}}_1 r^{\frac{1}{2}}_1
\sum\limits_{r_2 | (q_2, m, n)}q^{-\frac{1}{2}}_2 r^{\frac{1}{2}}_2\nonumber\\
&\leq \tau^2(q_1 q_2)\,(q_1 q_2)^{\frac{1}{2}} (q_1 q_2, m, n)^{\frac{1}{2}}\,.
\end{align}

\textbf{Case 2.} $q_1=2^h q'_1$, where $2\nmid q'_1$ and $h\leq2$, and $2\nmid q_2$.

The function $\lambda(q_1, q_2, m, n)$ defined by \eqref{lambdaq1q2mn} is such that,
if
\begin{equation*}
(q'_1 q''_1, q'_2 q''_2)=(q'_1, q''_1)=(q'_2, q''_2)=1
\end{equation*}
then
\begin{align}\label{multiplicative}
&\lambda(q'_1 q''_1, q'_2 q''_2, m, n)=\nonumber\\
&=\lambda\left(q'_1, q'_2, m \overline{(q''_1 q''_2)}_{q'_1 q'_2}, n \overline{(q''_1 q''_2)}_{q'_1 q'_2}\right)
\,\lambda\left(q''_1,q''_2, m \overline{(q'_1 q'_2)}_{q''_1 q''_2}, n \overline{(q'_1 q'_2)}_{q''_1 q''_2}\right)\,.
\end{align}
(Since the proof is elementary we skip the details and leave it to the reader.)

Using \eqref{lambdaq1q2mnest3}, \eqref{multiplicative} and the trivial estimate
$|\lambda(2^h, 1, m, n)|\leq4^h$ we find
\begin{align}\label{lambdaq1q2mnest4}
|\lambda(2^h q'_1, q_2, m, n)|
&=\left|\lambda\left(2^h, 1, m \overline{(q'_1 q_2)}_{2^h}, n \overline{(q'_1 q_2)}_{2^h}\right)
\,\lambda\left(q'_1, q_2, m \overline{(2^h)}_{q'_1 q_2}, n \overline{(2^h)}_{q'_1 q_2}\right)\right|\nonumber\\
&\leq 16\tau^2(q'_1 q_2)\,(q'_1 q_2)^{\frac{1}{2}} (q'_1 q_2, m, n)^{\frac{1}{2}}\nonumber\\
&\leq 16\tau^2(q_1 q_2)\,(q_1 q_2)^{\frac{1}{2}} (q_1 q_2, m, n)^{\frac{1}{2}}\,.
\end{align}

\textbf{Case 3.} $2\nmid q_1$ and $q_2=2^h q'_2$, where $2\nmid q'_2$ and $h\leq2$.

By \eqref{lambdaq1q2mnest3}, \eqref{multiplicative} and the trivial estimate
$|\lambda(1, 2^h, m, n)|\leq4^h$ we get
\begin{align}\label{lambdaq1q2mnest5}
|\lambda(q_1, 2^h q'_2, m, n)|
&=\left|\lambda\left(1, 2^h, m \overline{(q_1 q'_2)}_{2^h}, n \overline{(q_1 q'_2)}_{2^h}\right)
\,\lambda\left(q_1, q'_2, m \overline{(2^h)}_{q_1 q'_2}, n \overline{(2^h)}_{q_1 q'_2}\right)\right|\nonumber\\
&\leq 16\tau^2(q_1 q'_2)\,(q_1 q'_2)^{\frac{1}{2}} (q_1 q'_2, m, n)^{\frac{1}{2}}\nonumber\\
&\leq 16\tau^2(q_1 q_2)\,(q_1 q_2)^{\frac{1}{2}} (q_1 q_2, m, n)^{\frac{1}{2}}\,.
\end{align}

Now the estimation \eqref{Myestimate1} follows from
\eqref{lambdaq1q2mnest3}, \eqref{lambdaq1q2mnest4} and \eqref{lambdaq1q2mnest5}.

As a byproduct of \eqref{Myestimate1} we obtain \eqref{Myestimate2}.
\end{proof}

\begin{lemma}\label{Mylemma2}
Assume that $8\nmid q_1$, $8\nmid q_2$, $(q_1, q_2)=1$ and $H_0\geq2$. Then for the sums
\begin{equation}\label{Lambda12}
\Lambda_1=\sum\limits_{1\leq m\leq H_0}\frac{|\lambda(q_1, q_2, m, 0)|}{m}\,,\quad
\Lambda_2=\sum\limits_{1\leq m, n\leq H_0}\frac{|\lambda(q_1, q_2, m, n)|}{m n}
\end{equation}
the estimations
\begin{equation}\label{Lambda12est}
\Lambda_1\ll(q_1 q_2)^{\frac{1}{2}+\varepsilon} H^\varepsilon_0\,,\quad
\Lambda_2\ll(q_1 q_2)^{\frac{1}{2}+\varepsilon} H^\varepsilon_0
\end{equation}
hold.
\end{lemma}
\begin{proof}
Using \eqref{Lambda12} and Lemma \ref{Mylemma1} we get
\begin{equation}\label{Lambda1est1}
\Lambda_1\ll(q_1 q_2)^{\frac{1}{2}+\varepsilon}
\sum\limits_{1\leq m\leq H_0}\frac{(q_1 q_2, m)^{\frac{1}{2}}}{m}
=(q_1 q_2)^{\frac{1}{2}+\varepsilon}\Lambda_0\,,
\end{equation}
where
\begin{equation*}
\Lambda_0=\sum\limits_{1\leq m\leq H_0}\frac{(q_1 q_2, m)^{\frac{1}{2}}}{m}\,.
\end{equation*}
We have
\begin{equation}\label{Lambda0est}
\Lambda_0\ll\sum\limits_{r | q_1 q_2}r^{\frac{1}{2}}\sum\limits_{m\leq H_0\atop{m\equiv0\,(r)}}\frac{1}{m}
\ll(\log H_0)\sum\limits_{r | q_1 q_2}r^{-\frac{1}{2}}\ll(q_1 q_2 H_0)^\varepsilon\,.
\end{equation}
From \eqref{Lambda1est1} and \eqref{Lambda0est} follows the first inequality in \eqref{Lambda12est}.

Using \eqref{Lambda12}, \eqref{Lambda0est} and Lemma \ref{Mylemma1} we obtain
\begin{align*}
\Lambda_2&\ll(q_1 q_2)^{\frac{1}{2}+\varepsilon}
\sum\limits_{1\leq m, n\leq H_0}\frac{(q_1 q_2, m, n)^{\frac{1}{2}}}{m n}
\ll(q_1 q_2)^{\frac{1}{2}+\varepsilon}
\sum\limits_{1\leq m, n\leq H_0}\frac{(q_1 q_2, m)^{\frac{1}{2}} (q_1 q_2, n)^{\frac{1}{2}}}{m n}\\
&=(q_1 q_2)^{\frac{1}{2}+\varepsilon}\Lambda^2_0\ll(q_1 q_2)^{\frac{1}{2}+\varepsilon} H^\varepsilon_0\,,
\end{align*}
which proves the second inequality in \eqref{Lambda12est}.
\end{proof}

The final lemma we need gives us important expansions.
\begin{lemma}\label{expansion}
For any $H_0\geq2$, we have
\begin{equation*}
\psi(t)=-\sum\limits_{1\leq|m|\leq H_0}\frac{e(mt)}{2\pi i m}
+\mathcal{O}\big(f(H_0, t)\big)\,,
\end{equation*}
where $f(H_0, t)$ is a positive, infinitely many times differentiable and periodic with
period 1 function of $t$. It can be expanded into Fourier series
\begin{equation*}
f(H_0, t)=\sum\limits_{m=-\infty}^{+\infty}b_{H_0}(m)e(m t)\,,
\end{equation*}
with coefficients $b_{H_0}(m)$ such that
\begin{equation*}
b_{H_0}(m)\ll\frac{\log H_0}{H_0}\quad \mbox{for all}\quad m
\end{equation*}
and
\begin{equation*}
\sum\limits_{|m|>H^{1+\varepsilon}_0}|b_{H_0}(m)|\ll H^{-A}_0\,.
\end{equation*}
Here $A > 0$ is arbitrarily large and the constant in the $\ll$ - symbol depends on $A$ and $\varepsilon$.
\begin{proof}
See \cite{Tolev1}.
\end{proof}
\end{lemma}

\section{Proof of the theorem}
\indent

Using \eqref{GammaH} and the well-known identity $\mu^2(n)=\sum_{d^2|n}\mu(d)$ we get
\begin{equation}\label{GammaHdecomp}
\Gamma(H)=\sum\limits_{d_1, d_2\atop{(d_1, d_2)=1}}\mu(d_1)\mu(d_2)
\sum\limits_{1\leq x, y\leq H\atop{x^2+y^2+1\equiv 0\,(d_1^2)\atop{x^2+y^2+2\equiv 0\,(d_2^2)}}}1
=\Gamma_1(H)+\Gamma_2(H)\,,
\end{equation}
where
\begin{align}
\label{GammaH1}
&\Gamma_1(H)=\sum\limits_{d_1 d_2\leq z\atop{(d_1, d_2)=1}}\mu(d_1)\mu(d_2)\Sigma(H, d_1^2, d_2^2)\,,\\
\label{GammaH2}
&\Gamma_2(H)=\sum\limits_{d_1 d_2>z\atop{(d_1, d_2)=1}}\mu(d_1)\mu(d_2)\Sigma(H, d_1^2, d_2^2)\,,\\
\label{Sigma}
&\Sigma(H, d_1^2, d_2^2)=\sum\limits_{1\leq x, y\leq H\atop{x^2+y^2+1\equiv 0\,(d_1^2)\atop{x^2+y^2+2\equiv 0\,(d_2^2)}}}1\,,\\
\label{z}
&\sqrt{H}\leq z\leq H\,.
\end{align}

\textbf{Estimation of} $\mathbf{\Gamma_1(H)}$

At the estimation of $\Gamma_1(H)$ we will suppose that $q_1=d_1^2$, $q_2=d_2^2$, where
$d_1$ and $d_2$ are square-free, $(q_1, q_2)=1$ and $d_1 d_2\leq z$.

Denote
\begin{equation}\label{Omega}
\Omega(H, q_1, q_2, x)=\sum\limits_{h\leq H\atop{h\equiv x\,(q_1 q_2)}}1\,.
\end{equation}
Apparently
\begin{equation}\label{Omegaest}
\Omega(H, q_1, q_2, x)=Hq^{-1}_1 q^{-1}_2+\mathcal{O}(1)\,.
\end{equation}
Using \eqref{conditions}, \eqref{Sigma} and \eqref{Omega} we obtain
\begin{equation}\label{SigmaOmega}
\Sigma(H, q_1, q_2)=\sum\limits_{x, y\,:\,\eqref{conditions}}
\Omega(H, q_1, q_2, x)\,\Omega(H, q_1, q_2, y)\,.
\end{equation}
On the other hand \eqref{Omega} gives us
\begin{equation}\label{Omegapsi}
\Omega(H, q_1, q_2, y)=\left[\frac{H-y}{q_1 q_2}\right]-\left[\frac{-y}{q_1 q_2}\right]
=\frac{H}{q_1 q_2}+\psi\left(\frac{-y}{q_1 q_2}\right)
-\psi\left(\frac{H-y}{q_1 q_2}\right)\,.
\end{equation}
From \eqref{SigmaOmega} and \eqref{Omegapsi} we find
\begin{equation}\label{Sigmaest1}
\Sigma(H, q_1, q_2)=\sum\limits_{x, y\,:\,\eqref{conditions}}
\Omega(H, q_1, q_2, x)\left(\frac{H}{q_1 q_2}
-\psi\left(\frac{H-y}{q_1 q_2}\right)\right)+\Sigma'\,,
\end{equation}
where
\begin{equation*}
\Sigma'=\sum\limits_{x, y\,:\,\eqref{conditions}}
\Omega(H, q_1, q_2, x)\,\psi\left(\frac{-y}{q_1 q_2}\right)\,.
\end{equation*}
We shall estimate the sum $\Sigma'$. For the purpose we decompose it as
\begin{equation}\label{Sigma'decomp}
\Sigma'=\Sigma''+\Sigma'''\,,
\end{equation}
where
\begin{align}
\label{Sigma''}
&\Sigma''=\sum\limits_{1\leq x\leq q_1q_2\atop{x^2+1\equiv0\,(q_1)
\atop{x^2+2\equiv0\,(q_2)}}}\Omega(H, q_1, q_2, x)
\sum\limits_{1\leq y\leq q_1q_2\atop{y^2\equiv0\,(q_1 q_2)}}
\psi\left(\frac{-y}{q_1 q_2}\right)\,,\\
\label{Sigma'''}
&\Sigma'''=\sum\limits_{1\leq x\leq q_1q_2\atop{x^2+1\not\equiv0\,(q_1)
\atop{x^2+2\not\equiv0\,(q_2)}}}\Omega(H, q_1, q_2, x)
\sum\limits_{1\leq y\leq q_1q_2\atop{y^2\equiv-x^2-1\,(q_1)
\atop{y^2\equiv-x^2-2\,(q_2)}}}
\psi\left(\frac{-y}{q_1 q_2}\right)\,.
\end{align}

Firstly we consider the sum $\Sigma'''$.
We note that the sum over $y$ in  \eqref{Sigma'''} does not contain terms with
$y=\frac{q_1 q_2}{2}$ and $y=q_1 q_2$ .
Moreover for any $y$ satisfying the congruences and such that
$1\leq y<\frac{q_1 q_2}{2}$ the number $q_1 q_2-y$ satisfies the same congruences
and we have $\psi\left(\frac{-y}{q_1 q_2}\right)+\psi\left(\frac{-(q_1 q_2-y)}{q_1 q_2}\right)=0$.
Bearing in mind these arguments for the sum $\Sigma'''$ denoted by \eqref{Sigma'''}
we have that
\begin{equation}\label{Sigma'''est}
\Sigma'''=0\,.
\end{equation}

Next we consider the sum $\Sigma''$.
According to the above considerations
the sum over $y$ in \eqref{Sigma''} reduces to a sum with at most two terms
(corresponding to $y = \frac{q_1 q_2}{2}$ and $y = q_1 q_2)$.
Therefore
\begin{equation}\label{Sigma''est1}
\Sigma''\ll\sum\limits_{1\leq x\leq q_1q_2\atop{x^2+1\equiv0\,(q_1)
\atop{x^2+2\equiv0\,(q_2)}}}\Omega(H, q_1, q_2, x)\,.
\end{equation}
Now taking into account \eqref{Omegaest}, \eqref{Sigma''est1}, Chinese remainder theorem
and that the number of solutions of the congruence $x^2\equiv a\,(q_1q_2)$
equals to $\mathcal{O}((q_1q_2)^\varepsilon)=\mathcal{O}(H^\varepsilon)$ we get
\begin{equation}\label{Sigma''est2}
\Sigma''\ll \sum\limits_{1\leq x\leq q_1q_2\atop{x^2\equiv a\,(q_1q_2)}}\Omega(H, q_1, q_2, x)
\ll H^\varepsilon\Big(Hq^{-1}_1 q^{-1}_2+1\Big)\,.
\end{equation}
Here $a$ that depends on $q_1$ and $q_2$.\\
From \eqref{Sigma'decomp}, \eqref{Sigma'''est} and \eqref{Sigma''est2} it follows
\begin{equation}\label{Sigma'est}
\Sigma'\ll H^\varepsilon\Big(Hq^{-1}_1 q^{-1}_2+1\Big)\,.
\end{equation}
By \eqref{Sigmaest1} and \eqref{Sigma'est} we obtain
\begin{align}\label{Sigmaest2}
\Sigma(H, q_1, q_2)=&\sum\limits_{x, y\,:\,\eqref{conditions}}
\Omega(H, q_1, q_2, x)\left(\frac{H}{q_1 q_2}
-\psi\left(\frac{H-y}{q_1 q_2}\right)\right)\nonumber\\
&+\mathcal{O}\bigg(H^\varepsilon\Big(Hq^{-1}_1 q^{-1}_2+1\Big)\bigg)\,.
\end{align}
Proceeding in the same way with the sum $\Omega(H, q_1, q_2, x)$  we find
\begin{align}\label{Sigmaest3}
\Sigma(H, q_1, q_2)=&\sum\limits_{x, y\,:\,\eqref{conditions}}
\left(\frac{H}{q_1 q_2}-\psi\left(\frac{H-x}{q_1 q_2}\right)\right)
\left(\frac{H}{q_1 q_2}-\psi\left(\frac{H-y}{q_1 q_2}\right)\right)\nonumber\\
&+\mathcal{O}\bigg(H^\varepsilon\Big(Hq^{-1}_1 q^{-1}_2+1\Big)\bigg)\,.
\end{align}
Bearing in mind \eqref{lambdaq1q2mn}, \eqref{lambdaq1q2}
and \eqref{Sigmaest3} we get
\begin{align}\label{Sigmaest4}
\Sigma(H, q_1, q_2)&=\frac{H^2 \lambda(q_1, q_2)}{q^2_1 q^2_2}
-2\frac{H}{q_1 q_2}\Sigma_1(H, q_1, q_2)\nonumber\\
&+\Sigma_2(H, q_1, q_2)
+\mathcal{O}\bigg(H^\varepsilon\Big(Hq^{-1}_1 q^{-1}_2+1\Big)\bigg)\,,
\end{align}
where
\begin{align}
\label{Sigma1}
&\Sigma_1(H, q_1, q_2)=\sum\limits_{x, y\,:\,\eqref{conditions}}
\psi\left(\frac{H-x}{q_1 q_2}\right)\,,\\
\label{Sigma2}
&\Sigma_2(H, q_1, q_2)=\sum\limits_{x, y\,:\,\eqref{conditions}}
\psi\left(\frac{H-x}{q_1 q_2}\right)\, \psi\left(\frac{H-y}{q_1 q_2}\right)\,.
\end{align}

Firstly we consider the sum $\Sigma_1(H, q_1, q_2)$.
Using  \eqref{lambdaq1q2mn} and Lemma \ref{expansion} with $H_0=H$ we obtain
\begin{equation}\label{Sigma1decomp}
\Sigma_1(H, q_1, q_2)=\Sigma'_1(H, q_1, q_2)
+\mathcal{O}\Big(\Sigma''_1(H, q_1, q_2)\Big)\,,
\end{equation}
where
\begin{align}
\label{Sigma'1}
\Sigma'_1(H, q_1, q_2)&=-\sum\limits_{x, y\,:\,\eqref{conditions}}
\sum\limits_{1\leq|m|\leq H}
\frac{e\left(m\left(\frac{H-x}{q_1 q_2}\right)\right)}{2\pi i m}\nonumber\\
&=-\sum\limits_{1\leq|m|\leq H}
\frac{e\left(\frac{m H}{q_1 q_2}\right)\lambda(q_1, q_2, -m, 0)}{2\pi i m}\,,\\
\label{Sigma''1}
\Sigma''_1(H, q_1, q_2)&=\sum\limits_{x, y\,:\,\eqref{conditions}}
f\left(H,\frac{H-x}{q_1 q_2}\right)\,.
\end{align}
Formula  \eqref{Sigma'1} and Lemma \ref{Mylemma2} give us
\begin{equation}\label{Sigma'1est}
\Sigma'_1(H, q_1, q_2)\ll H^\varepsilon\, (q_1 q_2)^{\frac{1}{2}} \,.
\end{equation}
In order to estimate $\Sigma''_1(H, q_1, q_2)$  we use \eqref{lambdaq1q2mn}, \eqref{lambdaq1q2},
\eqref{Sigma''1} and Lemmas \ref{Mylemma1}, \ref{Mylemma2},  \ref{expansion}
and get
\begin{align}
\label{Sigma''1est}
\Sigma''_1(H, q_1, q_2)&=\sum\limits_{x, y\,:\,\eqref{conditions}}
\Bigg(b_H(0)+  \sum\limits_{1\leq|m|\leq H^{1+\varepsilon}}
b_H(m)\, e\left(m\left(\frac{H-x}{q_1 q_2}\right)\right)\Bigg)+\mathcal{O}(1)\nonumber\\
&=b_H(0)\, \lambda(q_1, q_2)+  \sum\limits_{1\leq|m|\leq H^{1+\varepsilon}}
b_H(m)\, e\left(\frac{mH}{q_1 q_2}\right)\lambda(q_1, q_2, -m, 0)+\mathcal{O}(1)\nonumber\\
&\ll H^{\varepsilon-1} \, q_1 q_2 +1+ H^{\varepsilon-1}\sum\limits_{1\leq|m|\leq H^{1+\varepsilon}}|\lambda(q_1, q_2, -m, 0)|\nonumber\\
&\ll H^{\varepsilon-1} \, q_1 q_2 +1+ H^\varepsilon\sum\limits_{1\leq m\leq H^{1+\varepsilon}}\frac{|\lambda(q_1, q_2, m, 0)|}{m}\nonumber\\
&\ll H^{\varepsilon-1} \, q_1 q_2 +H^\varepsilon \, (q_1q_2)^{\frac{1}{2}}\,.
\end{align}
From \eqref{Sigma1decomp}, \eqref{Sigma'1est} and \eqref{Sigma''1est}
it follows
\begin{equation}\label{Sigma1est}
\Sigma_1(H, q_1, q_2)\ll H^{\varepsilon-1} \, q_1 q_2
+H^\varepsilon \, (q_1q_2)^{\frac{1}{2}}\,.
\end{equation}

Next we consider the sum $\Sigma_2(H, q_1, q_2)$.
Bearing in mind \eqref{Sigma2}, \eqref{Sigma''1}, \eqref{Sigma''1est}
and Lemmas \ref{Mylemma2},  \ref{expansion} we find
\begin{align}
\label{Sigma2est}
\Sigma_2(H, q_1, q_2)&=\sum\limits_{x, y\,:\,\eqref{conditions}}
\sum\limits_{1\leq |m|, |n| \leq H}
\frac{e\left(\frac{(m+n)H}{q_1 q_2}\right)\, e\left(-\frac{m x+n y}{q_1 q_2}\right)}
{(2\pi i)^2 m n}+\mathcal{O}\Big(H^\varepsilon \Sigma''_1(H, q_1, q_2)\Big)\nonumber\\
&=\sum\limits_{1\leq |m|, |n| \leq H}
\frac{e\left(\frac{(m+n)H}{q_1 q_2}\right)}{(2\pi i)^2 m n}\lambda(q_1, q_2, -m, -n)
+\mathcal{O}\Big(H^{\varepsilon-1} \, q_1 q_2 +H^\varepsilon \, (q_1q_2)^{\frac{1}{2}}\Big)\nonumber\\
&\ll\sum\limits_{1\leq |m|, |n| \leq H}
\frac{|\lambda(q_1, q_2, m, n)|}{|m n|}
+H^{\varepsilon-1} \, q_1 q_2 +H^\varepsilon \, (q_1q_2)^{\frac{1}{2}}\nonumber\\
&\ll H^{\varepsilon-1} \, q_1 q_2 +H^\varepsilon \, (q_1q_2)^{\frac{1}{2}}\,.
\end{align}
Taking into account \eqref{Sigmaest4}, \eqref{Sigma1est} and \eqref{Sigma2est}
we get
\begin{equation}\label{Sigmaest5}
\Sigma(H, q_1, q_2)=H^2 \frac{ \lambda(q_1, q_2)}{q^2_1 q^2_2}
+\mathcal{O}\bigg(H^\varepsilon\Big(Hq^{-\frac{1}{2}}_1 q^{-\frac{1}{2}}_2
+q^{\frac{1}{2}}_1 q^{\frac{1}{2}}_2+H^{-1} q_1 q_2\Big)\bigg)\,.
\end{equation}
From \eqref{GammaH1}, \eqref{z} and \eqref{Sigmaest5} we obtain
\begin{align}\label{GammaH1est1}
\Gamma_1(H)&=H^2 \sum\limits_{d_1 d_2\leq z\atop{(d_1, d_2)=1}}
 \frac{\mu(d_1)\mu(d_2) \lambda(d^2_1, d^2_2)}{d^4_1 d^4_2}
+\mathcal{O}\big(H^\varepsilon z^2\big)\nonumber\\
&=\sigma H^2-H^2 \sum\limits_{d_1 d_2>z\atop{(d_1, d_2)=1}}
 \frac{\mu(d_1)\mu(d_2) \lambda(d^2_1, d^2_2)}{d^4_1 d^4_2}
+\mathcal{O}\big(H^\varepsilon z^2\big)\,,
\end{align}
where
\begin{equation}\label{sigmasum}
\sigma=\sum\limits_{d_1, d_2=1\atop{(d_1, d_2)=1}}^\infty
 \frac{\mu(d_1)\mu(d_2)\lambda(d^2_1, d^2_2)}{d^4_1 d^4_2}\,.
\end{equation}
Using \eqref{Myestimate2} we find
\begin{equation}\label{d1d2>est}
\sum\limits_{d_1 d_2>z\atop{(d_1, d_2)=1}}\frac{\mu(d_1)\mu(d_2)\lambda(d^2_1, d^2_2)}{d^4_1 d^4_2}
\ll\sum\limits_{d_1 d_2>z\atop{(d_1, d_2)=1}}\frac{(d_1 d_2)^{2+\varepsilon}}{(d_1 d_2)^4}
\ll\sum\limits_{n>z}\frac{\tau(n)}{n^{2-\varepsilon}}\ll z^{\varepsilon-1}\,.
\end{equation}
It remains to see that the product  \eqref{sigmaproduct} and the sum \eqref{sigmasum} coincide.
From the definition \eqref{lambdaq1q2}, the property \eqref{multiplicative} and  $(d_1, d_2)=1$  it follows
\begin{equation}\label{lambdad1d2}
\lambda(d^2_1, d^2_2)=\lambda(d^2_1, 1) \lambda(1, d^2_2)\,.
\end{equation}
Bearing in mind \eqref{sigmasum} and \eqref{lambdad1d2} we get
\begin{equation}\label{sigmasumest1}
\sigma=\sum\limits_{d_1=1}^\infty\frac{\mu(d_1)\lambda(d^2_1, 1)}{d_1^4}
\sum\limits_{d_2=1}^\infty\frac{\mu(d_2)\lambda(1, d^2_2)}{d_2^4}f_{d_1}(d_2)\,,
\end{equation}
where
\begin{equation*}
f_{d_1}(d_2)=\begin{cases}1 \;\; \text{ if }\; (d_1, d_2)=1\,,\\
0 \;\; \mbox{ if } \; (d_1, d_2)>1\,.
\end{cases}
\end{equation*}
Clearly the function
\begin{equation*}
\frac{\mu(d_2)\lambda(1, d^2_2)}{d_2^4}f_{d_1}(d_2)
\end{equation*}
is multiplicative with respect to $d_2$ and the series
\begin{equation*}
\sum\limits_{d_2=1}^\infty\frac{\mu(d_2)\lambda(1, d^2_2)}{d_2^4}f_{d_1}(d_2)
\end{equation*}
is absolutely convergent.

Applying the Euler product we obtain
\begin{align}\label{Eulerproduct}
\sum\limits_{d_2=1}^\infty\frac{\mu(d_2)\lambda(1, d^2_2)}{d_2^4}f_{d_1}(d_2)&=
\prod\limits_{p\nmid d_1}\left(1-\frac{\lambda(1, p^2)}{p^4}\right)\nonumber\\
&=\prod\limits_{p}\left(1-\frac{\lambda(1, p^2)}{p^4}\right)
\prod\limits_{p|d_1}\left(1-\frac{\lambda(1, p^2)}{p^4}\right)^{-1}\,.
\end{align}
From \eqref{sigmasumest1} and \eqref{Eulerproduct} it follows
\begin{align}\label{sigmasumest2}
\sigma&=\sum\limits_{d_1=1}^\infty\frac{\mu(d_1)\lambda(d^2_1, 1)}{d_1^4}
\prod\limits_{p}\left(1-\frac{\lambda(1, p^2)}{p^4}\right)
\prod\limits_{p|d_1}\left(1-\frac{\lambda(1, p^2)}{p^4}\right)^{-1}\nonumber\\
&=\prod\limits_{p}\left(1-\frac{\lambda(1, p^2)}{p^4}\right)
\sum\limits_{d_1=1}^\infty\frac{\mu(d_1)\lambda(d^2_1, 1)}{d_1^4}
\prod\limits_{p|d_1}\left(1-\frac{\lambda(1, p^2)}{p^4}\right)^{-1}\,.
\end{align}
Obviously  the function
\begin{equation*}
\frac{\mu(d_1)\lambda(d^2_1, 1)}{d_1^4}
\prod\limits_{p|d_1}\left(1-\frac{\lambda(1, p^2)}{p^4}\right)^{-1}
\end{equation*}
is multiplicative with respect to $d_1$ and the series
\begin{equation*}
\sum\limits_{d_1=1}^\infty\frac{\mu(d_1)\lambda(d^2_1, 1)}{d_1^4}
\prod\limits_{p|d_1}\left(1-\frac{\lambda(1, p^2)}{p^4}\right)^{-1}
\end{equation*}
is absolutely convergent.

Applying again the Euler product from \eqref{sigmasumest2} we find
\begin{align}\label{sigmasumest3}
\sigma&=
\prod\limits_{p}\left(1-\frac{\lambda(1, p^2)}{p^4}\right)
\prod\limits_{p}\left(1-\frac{\lambda(p^2, 1)}{p^4}\left(1-\frac{\lambda(1, p^2)}{p^4}\right)^{-1}\right)\nonumber\\
&=\prod\limits_{p}\left(1-\frac{\lambda(p^2, 1)+\lambda(1, p^2)}{p^4}\right)\,.
\end{align}
Bearing in mind \eqref{z}, \eqref{GammaH1est1}, \eqref{d1d2>est} and \eqref{sigmasumest3} we get
\begin{equation}\label{GammaH1est2}
\Gamma_1(H)=\sigma H^2+\mathcal{O}\Big(H^\varepsilon\big(z^2+H^2z^{-1}\big)\Big)\,,
\end{equation}
where $\sigma$ is given by the product \eqref{sigmaproduct}.

\textbf{Estimation of} $\mathbf{\Gamma_2(H)}$

Using \eqref{GammaH2} we write
\begin{equation}\label{GammaH2est1}
|\Gamma_2(H)|
\ll(\log H)^2\sum\limits_{D_1\leq d_1<2D_1}\sum\limits_{D_2\leq d_2<2D_2}
\sum\limits_{k\leq(2H^2+1)d_1^{-2}\atop{kd_1^2+1\equiv0\,(d_2^2)}}
\sum\limits_{1\leq x, y\leq H\atop{x^2+y^2=kd_1^2-1}}1\,,
\end{equation}
where
\begin{equation}\label{DT}
\frac{1}{2}\leq D_1, D_2\leq\sqrt{2H^2+2}\,,\quad D_1D_2\geq\frac{z}{4}\,.
\end{equation}
On the one hand \eqref{GammaH2est1} give us
\begin{align}\label{GammaH2est2}
|\Gamma_2(H)|&\ll H^\varepsilon\sum\limits_{D_1\leq d_1<2D_1}
\sum\limits_{k\leq(2H^2+1)D_1^{-2}}\sum\limits_{D_2\leq d_2<2D_2}
\sum\limits_{l\leq(2H^2+2)D_2^{-2}\atop{kd_1^2+1=ld_2^2}}1\nonumber\\
&\ll  H^\varepsilon\sum\limits_{D_1\leq d_1<2D_1}
\sum\limits_{k\leq(2H^2+1)D_1^{-2}}\tau(kd_1^2+1)\nonumber\\
&\ll H^\varepsilon\sum\limits_{D_1\leq d_1<2D_1}
\sum\limits_{k\leq(2H^2+1)D_1^{-2}}1\nonumber\\
&\ll H^{2+\varepsilon}D_1^{-1}\,.
\end{align}
On the other hand \eqref{GammaH2est1} implies
\begin{align}\label{GammaH2est3}
|\Gamma_2(H)|&\ll H^\varepsilon\sum\limits_{D_2\leq d_2<2D_2}
\sum\limits_{l\leq(2H^2+2)D_2^{-2}}\sum\limits_{D_1\leq d_1<2D_1}
\sum\limits_{k\leq(2H^2+1)D_1^{-2}\atop{kd_1^2=ld_2^2-1}}1\nonumber\\
&\ll H^\varepsilon\sum\limits_{D_2\leq d_2<2D_2}
\sum\limits_{l\leq(2H^2+2)D_2^{-2}}\tau(ld_2^2-1)\nonumber\\
&\ll H^\varepsilon\sum\limits_{D_2\leq d_2<2D_2}
\sum\limits_{l\leq(2H^2+2)D_2^{-2}}1\nonumber\\
&\ll H^{2+\varepsilon}D_2^{-1}\,.
\end{align}
By \eqref{DT} -- \eqref{GammaH2est3} it follows
\begin{equation}\label{GammaH2est4}
|\Gamma_2(H)|\ll  H^{2+\varepsilon}z^{-\frac{1}{2}}\,.
\end{equation}

\textbf{The end of the proof}

Bearing in mind \eqref{GammaHdecomp}, \eqref{GammaH1est2}, \eqref{GammaH2est4}
and choosing $z=H^{\frac{4}{5}}$ we obtain the asymptotic formula
\eqref{asymptoticformula}.

The theorem is proved.
\vspace{5mm}

\textbf{Acknowledgments.}
The author is especially grateful to Prof. Tolev for numerous directions
and advices for five years.
The door to Prof. Tolev office
was always open whenever I had a question about my research.
This paper is dedicated to Tolev's birthday anniversary.

\vskip20pt
\footnotesize
\begin{flushleft}
S. I. Dimitrov\\
Faculty of Applied Mathematics and Informatics\\
Technical University of Sofia \\
8, St.Kliment Ohridski Blvd. \\
1756 Sofia, BULGARIA\\
e-mail: sdimitrov@tu-sofia.bg\\
\end{flushleft}

\begin{thebibliography}{}

\bibitem{Carlitz}L. Carlitz, {\it On a problem in additive arithmetic II},
Quart. J. Math., {\bf3}, (1932), 273 -- 290.

\bibitem{Dimitrov1}S. I.  Dimitrov, {\it Consecutive square-free numbers of the form $[n^c], [n^c]+1$},
JP Journal of Algebra, Number Theory and Applications, \textbf{40}, 6, (2018), 945 -- 956.

\bibitem{Dimitrov2} S. I. Dimitrov, {\it On the distribution of consecutive square-free
numbers of the form  $[\alpha n], [\alpha n]+1$},
arXiv:1903.04545v2 [math.NT] 22 Mar 2019.

\bibitem{Dimitrov3} S. I. Dimitrov, {\it  Consecutive square-free values of the form $[\alpha p], [\alpha p]+1$},
arXiv:1907.03721v1  [math.NT]  3 Jul 2019.

\bibitem{Estermann} T. Estermann, {\it A new application of the Hardy-Littlewood-Kloosterman method},
Proc. London Math. Soc., \textbf{12}, (3), (1962), 425 -- 444.

\bibitem{Heath-Brown}D. R. Heath-Brown, {\it The Square-Sieve and Consecutive Square-Free Numbers},
Math. Ann., {\bf266}, (1984), 251 -- 259.

\bibitem{Hua} L. K. Hua, {\it Introduction to Number Theory}, Springer, Berlin, (1982).

\bibitem{Iwaniec} H. Iwaniec, E. Kowalski, {\it Analytic number theory},
Colloquium Publications, {\bf53}, Am. Math. Soc., (2004).

\bibitem{Reuss}T. Reuss, {\it Pairs of k-free Numbers, consecutive square-full Numbers},
arXiv:1212.3150v2 [math.NT].

\bibitem{Tolev1} D. I. Tolev,  {\it On the exponential sum with squarefree numbers},
Bull. Lond. Math. Soc., {\bf37}, 6, (2005), 827 -- 834.

\bibitem{Tolev2} D. I. Tolev,  {\it On the number of pairs of positive integers
$x, y \leq H$ such that $x^2+y^2+1$ is squarefree},
Monatsh. Math., {\bf165}, (2012), 557 -- 567.

\end{thebibliography}
\end{document}